\documentclass[a4paper,11pt]{article}
\usepackage[english]{babel}
\usepackage[T1]{fontenc}
\usepackage[utf8]{inputenc}
\usepackage{enumerate}
\usepackage{amsmath}
\usepackage{amsthm}
\usepackage{amssymb}
\usepackage{mathtools}
\usepackage{epsfig}
\usepackage{color}
\usepackage{csquotes}
\usepackage[backend=biber,
            style=alphabetic,
            giveninits=true,
            isbn=false,
            doi=false,
            url=false,
            sorting=nyvt]{biblatex}
\renewbibmacro{in:}{}

\newtheorem*{thm*}{Theorem} 
\newtheorem{thm}{Theorem} 
\newtheorem{thmA}{Theorem} 
\newtheorem*{lemma*}{Lemma}
\newtheorem{lemma}{Lemma}
\newtheorem*{prop*}{Proposition}

\newtheorem*{corl*}{Corollary}
\newtheorem{corl}{Corollary}

\theoremstyle{definition}

\newcommand{\dist}{\operatorname{dist}}

\newcommand{\integers}{\mathbb{Z}}
\newcommand{\disk}{\mathbb{D}}

\newcommand{\reals}{\mathbb{R}}

\newcommand{\Lp}[1]{L^{#1}}
\newcommand{\BMO}{\operatorname{BMO}}
\newcommand{\continuous}{\mathcal{C}}
\newcommand{\zyg}{\Lambda_{\ast}}
\newcommand{\zygdyc}{\Lambda_{\ast d}}
\newcommand{\IBMO}{\operatorname{I}(\BMO)}
\newcommand{\Lip}{\operatorname{Lip}}
\newcommand{\clIBMO}{\overline{\IBMO}}
\newcommand{\sobolev}[2]{W^{#1,#2}}
\newcommand{\norm}[2]{\left\lVert #1 \right\rVert_{#2}}

\title{Approximation in the Zygmund Class}
\author{Artur~Nicolau
        and Odí~Soler~i~Gibert
        \thanks{Both authors supported by the Generalitat de Catalunya (grant 2017 SGR 395) and the Spanish Ministerio de Ciencia e Innovaci\'on (projects MTM2014-51824-P and MTM2017-85666-P).}}
\date{}

\newcommand{\Addresses}{{
  \bigskip
  \footnotesize

  Artur~Nicolau: \textsc{Universitat Autònoma De Barcelona, Departament de Matemàtiques, Edifici C, 08193-Bellaterra, Catalunya}\par\nopagebreak
  \textit{E-mail address}: \texttt{artur@mat.uab.cat}

  \medskip

  Odí~Soler~i~Gibert: \textsc{Universitat Autònoma De Barcelona, Departament de Matemàtiques, Edifici C, 08193-Bellaterra, Catalunya}\par\nopagebreak
  \textit{E-mail address}: \texttt{odisoler@mat.uab.cat}

}}

\begin{document}

  \maketitle
  
  \begin{abstract}
    We study the distance in the Zygmund class $\zyg$ to the subspace $\IBMO$ of functions with distributional derivative with bounded mean oscillation.
    In particular, we describe the closure of $\IBMO$ in the Zygmund seminorm.
    We also generalise this result to Zygmund measures on $\reals^d.$
    Finally, we apply the techniques developed in the article to characterise the closure of the subspace of functions in $\zyg$ that are also in the classical Sobolev space $\sobolev{1}{p},$ for $1 < p < \infty.$
  \end{abstract}
  
  \section{Introduction}
  \label{sec:intro}
  
  A continuous real valued function $f$ on the real line belongs to the \emph{Zygmund class} $\zyg$ if 
  \begin{equation*}
    \norm{f}{\ast} = \sup_{\substack{x,h \in \reals\\ h > 0}} |\Delta_2f(x,h)| < \infty,
  \end{equation*}
  where
  \begin{equation*}
    \Delta_2f(x,h) = \frac{f(x+h) - 2f(x) + f(x-h)}{h}
  \end{equation*}
  denotes the \emph{second divided difference} centred at $x$ with step $h,$ or in other words, the second divided difference on the interval $I = (x-h,x+h)$ and denoted $\Delta_2f(I) = \Delta_2f(x,h).$
  For a function $f \in \zyg,$ the quantity $\norm{f}{\ast}$ is called the \emph{Zygmund seminorm} of $f.$
  The Zygmund class is the natural substitute of the space of Lipschitz functions in many different contexts as polynomial approximation, Bessel potentials, Calderón-Zygmund theory, and has been extensively studied (see for instance \cite{ref:Zygmund}, Chapter V of \cite{ref:Stein}, \cite{ref:MakarovSmoothMeasures}, \cite{ref:DonaireLlorenteNicolau}).
  
  For a measurable set $A \subset \reals,$ we denote by $|A|$ its Lebesgue measure, and we will denote by $\chi_A$ its indicator function.
  We use the standard notation $a \lesssim b$ (respectively $a \gtrsim b)$ if there exists an absolute constant $C > 0$ such that $a \leq Cb$ (resp. $a \geq Cb).$
  We will also denote $a \simeq b$ if $a \lesssim b$ and $a \gtrsim b.$
  
  A locally integrable function $f$ on the real line is said to have \emph{bounded mean oscillation,} $f \in \BMO,$ if
  \begin{equation}
    \label{eq:BMONorm}
    \norm{f}{\BMO} = \sup_I \left(\frac{1}{|I|} \int_I |f(x)-f_I|^2 \, dx\right)^{1/2} < \infty,
  \end{equation}
  where $I$ ranges over all finite intervals in $\reals$ and where
  \begin{equation*}
    f_I = \frac{1}{|I|} \int_I f(x)\, dx,
  \end{equation*}
  is the average of $f$ on $I.$
  The space of continuous functions such that their derivatives, in the sense of distributions, are $\BMO$ functions is
  \begin{equation*}
    \IBMO = \{f \in \continuous(\reals)\colon f'\in\BMO\}.
  \end{equation*}
  It is easy to check that $\IBMO \subsetneq \zyg.$
  In \cite{ref:Strichartz}, R.~Strichartz found a characterisation for functions in $\IBMO$ in terms of their second divided differences.
  We state it below for compactly supported functions.
  \begin{thmA}[R.~Strichartz]
    \label{thm:Strichartz}
    A compactly supported function $f$ is in $\IBMO$ if and only if
    \begin{equation}
      \label{eq:IBMOCondition}
      \sup_I \frac{1}{|I|} \int_I\int_0^{|I|} |\Delta_2f(x,h)|^2 \frac{dh\, dx}{|h|} < \infty,
    \end{equation}
    where $I$ ranges over all finite intervals on $\reals.$
  \end{thmA}

  One of the main goals of this article is to give an analog of Theorem \ref{thm:Strichartz} for functions in the closure $\clIBMO$ in the Zygmund seminorm $\norm{\cdot}{\ast},$ for which we will consider the pseudometric $\dist(f,g) = \norm{f-g}{\ast}$ for any pair of functions $f,g \in \zyg.$
  To this end, from now on for a given function $f \in \zyg$ and $\varepsilon > 0,$ consider the set
  \begin{equation*}
    A(f,\varepsilon) = \{(x,h) \in \reals^2_{+} \colon |\Delta_2f(x,h)| > \varepsilon\},
  \end{equation*}
  where we use $\reals^2_+$ to denote the upper halfplane $\reals^2_+ = \{(x,h)\colon x \in \reals, h > 0\}.$
  \begin{thm}
    \label{thm:IBMODistance}
    Let $f$ be a compactly supported function in $\zyg.$
    For each $\varepsilon > 0,$ consider
    \begin{equation*}
      C(f,\varepsilon) = \sup_{I} \frac{1}{|I|} \int_I\int_0^{|I|} \chi_{A(f,\varepsilon)}(x,h)\frac{dh\, dx}{h},
    \end{equation*}
    where $I$ ranges over all finite intervals.
    Then,
    \begin{equation}
      \label{eq:IBMODistance}
      \dist(f,\IBMO) \simeq \inf \{\varepsilon > 0 \colon C(f,\varepsilon) < \infty\}.
    \end{equation}
  \end{thm}
  
  We deduce the following description of $\clIBMO.$
  
  \begin{corl}
    Let $f$ be a compactly supported function in $\zyg.$
    Then $f \in \clIBMO$ if and only if for every $\varepsilon > 0$ there exists a constant $C(\varepsilon) > 0$ such that
    \begin{equation*}
      \frac{1}{|I|} \int_I\int_0^{|I|} \chi_{A(f,\varepsilon)}(x,h) \frac{dh\, dx}{h} \leq C(\varepsilon),
    \end{equation*}
    for every finite interval $I.$
  \end{corl}

  Observe that Theorem \ref{thm:IBMODistance} is actually a local result, and in this sense it can still be applied to functions that are not compactly supported by restricting to a finite interval.
  Hence, these results also hold for functions defined on the unit circle.
  It is worth mentioning that, for functions defined on the unit circle, the closure of the trigonometric polynomials in the Zygmund seminorm is the small Zygmund class (see \cite{ref:Zygmund}).
  Observe as well that Theorem \ref{thm:IBMODistance} also implies uniform approximation locally in the following sense.
  It is a well known fact (see for instance \cite{ref:JonssonWallin}) that for any function $f \in \zyg,$ and for any finite interval $I \subseteq \reals,$ there exists a polynomial $p_I$ of degree $1$ such that
  \begin{equation*}
    |f(x)-p_I(x)| \lesssim |I| \norm{f}{\ast},\quad x \in I.
  \end{equation*}
  Thus, if $f \in \zyg$ is compactly supported on an interval $I_0,$ there is $g \in \IBMO$ such that for any interval $I \subseteq I_0$ there exists a linear polynomial $p_I$ with
  \begin{equation*}
    |f(x)-(g+p_I)(x)| \lesssim |I| \dist(f,\IBMO),\quad x \in I.
  \end{equation*}

  The lower bound in \eqref{eq:IBMODistance} is easy, and the main part of the paper is devoted to prove the upper bound.
  We will first introduce a dyadic version of the Zygmund class, $\BMO$ and $\IBMO,$ and the corresponding notion for dyadic martingales.
  Then we state and prove a discrete version of \eqref{eq:IBMODistance}.
  Afterwards, an averaging argument of J.~Garnett and P.~Jones (see \cite{ref:GarnettJonesBMOAndDyadicBMO}) is used to prove the continuous result from the dyadic one.
  To this end, certain technical estimates are needed, which we have collected in Section \ref{sec:Preliminaries}.

  For $n \geq 0,$ let $\mathcal{D}_n = \{[k2^{-n},(k+1)2^{-n})\colon k \in \integers\}$ be the collection of dyadic intervals of length $2^{-n}.$
  For $n < 0,$ consider $m$ such that $n = -2m+1$ or $n = -2m,$ and let $t_n = (4^m-1)/3.$
  In this case, define $\mathcal{D}_n = \{[k2^{-n}-t_n,(k+1)2^{-n}-t_n)\colon k \in \integers\}.$
  Denote by $\mathcal{D} = \bigcup_{n \in \integers} \mathcal{D}_n.$
  We will call the intervals in $\mathcal{D}$ \emph{dyadic intervals.} This definition might look unnecessarily complicated for the dyadic intervals with $n < 0,$ where we add a translation by $t_n$ units with respect to the previous ones, but it will turn out to be convenient later on.
  The reason is that with this choice any finite interval $I \subset \reals$ is contained in some interval of $\mathcal{D},$ which is not true if we do not include any such translations.

  A locally integrable function $f$ has \emph{dyadic bounded mean oscillation,} $f \in \BMO_d,$ if condition \eqref{eq:BMONorm} is required only for dyadic intervals, that is, if
  \begin{equation*}
    \norm{f}{\BMO d} = \sup_{I \in \mathcal{D}} \left(\frac{1}{|I|} \int_I |f(x)-f_I|^2 \, dx\right)^{1/2} < \infty.
  \end{equation*}
  Note that $\BMO \subset \BMO_d.$
  The space $\BMO_d$ has been studied as a natural discrete substitute of $\BMO$ (see, for instance, \cite{ref:GarnettJonesBMOAndDyadicBMO}, \cite{ref:Mei} and \cite{ref:Conde}).
  The following result is stated in \cite{ref:GarnettJonesBMOAndDyadicBMO} and summarises the averaging technique previously mentioned.

  \begin{thmA}[J.~Garnett, P.~Jones]
    \label{thm:DyadicBMOtoBMO}
    Suppose that $\alpha \mapsto b^{(\alpha)}$ is a measurable mapping from $\reals$ to $BMO_d$ such that all $b^{(\alpha)}$ are supported on a fixed dyadic interval $I_0,$ and such that for every $\alpha,$ $\norm{b^{(\alpha)}}{BMO d} \leq 1$ and
    \begin{equation*}
      \int_\reals b^{(\alpha)}(x)\, dx = 0.
    \end{equation*}
    Then
    \begin{equation*}
      b_R(x) = \frac{1}{2R} \int_{-R}^R b^{(\alpha)}(x+\alpha)\, d\alpha
    \end{equation*}
    is in $BMO$ and there is a constant $C > 0$ such that $\norm{b_R}{BMO} \leq C$ for any $R \geq 1.$
  \end{thmA}
  
  We shall need an analogous result for the Zygmund class.
  We say that a continuous function $f$ belongs to the \emph{dyadic Zygmund class,} $f \in \zygdyc,$ if
  \begin{equation*}
    \norm{f}{\ast d} = \sup_{I \in \mathcal{D}} |\Delta_2f(I)| < +\infty.
  \end{equation*}
  Observe as well that $\zyg \subsetneq \zygdyc.$
  
  \begin{thm}
    \label{thm:zygAverages}
    Suppose that $\alpha \mapsto t^{(\alpha)}$ is a measurable mapping from $\reals$ to $\zygdyc$ such that all $t^{(\alpha)}$ are supported on a fixed dyadic interval $I_0,$ and such that for every $\alpha,$ $\norm{t^{(\alpha)}}{\ast d} \leq 1.$
    Then, the function
    \begin{equation*}
      t_R(x) = \frac{1}{2R} \int_{-R}^R t^{(\alpha)}(x+\alpha)\, d\alpha, \quad x\in\reals
    \end{equation*}
    is in $\zyg$ and there is a constant $C > 0$ such that $\norm{t_R}{\ast} \leq C$ for any $R \geq 1.$
  \end{thm}
  
  As an application of the techniques exposed in the article, we also show a result similar to Theorem \ref{thm:IBMODistance} for Sobolev spaces.
  For $1 < p < \infty,$ we consider the Sobolev space $\sobolev{1}{p}$ of functions $f \in \Lp{p}$ whose derivative $f'$ in the distributional sense is also in $\Lp{p}.$
  Take then the subspace of the Zygmund class $\zyg^p = \sobolev{1}{p} \cap \zyg.$
  The next theorem gives estimates for distances to this subspace.
  Here, for $x\in\reals,$ $\Gamma(x)$ denotes the truncated cone defined as $\Gamma(x) = \{(t,h) \in \reals^2_+ \colon |x-t| < h,\, 0 < h < 1\}.$
  \begin{thm}
    \label{thm:sobolev}
    Let $f$ be a compactly supported function in $\zyg.$
    For each $\varepsilon > 0,$ define the function
    \begin{equation*}
      C(f,\varepsilon)(x) = \left( \int_{\Gamma(x)} \chi_{A(f,\varepsilon)}(s,t) \frac{ds\, dt}{t^2} \right)^{1/2},
      \quad x \in \reals.
    \end{equation*}
    Then,
    \begin{equation}
      \label{eq:SobolevDistance}
      \dist(f,\zyg^p) \simeq \inf \{\varepsilon > 0 \colon C(f,\varepsilon) \in \Lp{p}\}.
    \end{equation}
  \end{thm}

  Finally, we find a higher dimensional analog of Theorem \ref{thm:IBMODistance} for Zygmund measures in $\reals^d.$
  Recall that a signed Borel measure $\mu$ on $\reals^d$ is called a Zygmund measure if
  \begin{equation*}
    \norm{\mu}{\ast} = \sup_Q \left| \frac{\mu(Q)}{|Q|} - \frac{\mu(Q^\ast)}{|Q^\ast|}  \right| < \infty,
  \end{equation*}
  where $Q$ ranges over all finite cubes in $\reals^d$ with edges parallel to the axis, and where $Q^\ast$ denotes the cube with the same centre as $Q$ but double side length.
  In the case $d = 1$ it is obvious that $\mu$ is a Zygmund measure if and only if its primitive $f(x) = \mu([0,x])$ is in the Zygmund class.
  Note that there exist Zygmund measures that are singular with respect to the Lebesgue measure, as J.~P.~Kahane showed \cite{ref:Kahane}.
  More information on Zygmund measures can be found in \cite{ref:MakarovSmoothMeasures}, \cite{ref:AndersonPitt} and \cite{ref:AleksandrovAndersonNicolau}.
  We consider the space of absolutely continuous measures $\nu$ such that $d\nu(x) = f(x)\, dx$ for some $f \in \BMO(\reals^d).$
  We call this the space of $\IBMO$ measures.
  It is clear that a measure in $\IBMO$ is a Zygmund measure as well.
  As before, given a Zygmund measure $\mu$ on $\reals^d,$ we want to describe the distance
  \begin{equation*}
    \dist(\mu,\IBMO) = \inf \{\norm{\mu-\sigma}{\ast}\colon \sigma \in \IBMO\}.
  \end{equation*}
  For $x \in \reals^d$ and $h > 0,$ let $Q(x,h)$ be the cube centred at $x$ of sidelength $h.$
  For a given Zygmund measure $\mu$ and for $\varepsilon > 0,$ consider the set
  \begin{equation*}
    A(\mu,\varepsilon) = \left\{ (x,h) \in \reals^{d+1}_+ \colon \left| \frac{\mu(Q(x,h))}{|Q(x,h)|} - \frac{\mu(Q(x,2h))}{|Q(x,2h)|} \right| > \varepsilon \right\},
  \end{equation*}
  where we use $\reals^{d+1}_+$ to denote the upper halfspace $\reals^{d+1}_+ = \{(x,h)\colon x \in \reals^d, h > 0\}.$
  \begin{thm}
    \label{thm:MeasureIBMODistance}
    Let $\mu$ be a compactly supported Zygmund measure on $\reals^d.$
    For each $\varepsilon > 0,$ consider
    \begin{equation*}
      C(\mu,\varepsilon) = \sup_{Q} \frac{1}{|Q|} \int_Q \int_0^{l(Q)} \chi_{A(\mu,\varepsilon)}(x,h)\, \frac{dh\, dx}{h},
    \end{equation*}
    where $Q$ ranges over all finite cubes and $l(Q)$ denotes the side length of $Q.$
    Then,
    \begin{equation*}
      \dist(\mu,\IBMO) \simeq \inf\{ \varepsilon > 0\colon C(\mu,\varepsilon) < \infty \}.
    \end{equation*}
  \end{thm}

  This paper is organised in the following manner.
  In Section \ref{sec:Preliminaries}, we expose the technical estimates that we need in order to apply the averaging argument previously mentioned.
  We then state and prove the dyadic analog of Theorem \ref{thm:IBMODistance} in Section \ref{sec:Dyadic}.
  In Section \ref{sec:Continuous}, we explain the averaging argument that yields Theorem \ref{thm:zygAverages} and then we use it to prove Theorem \ref{thm:IBMODistance}.
  Next, we explain in Section \ref{sec:Measures} the variations in the previous construction that allow us to prove Theorem \ref{thm:MeasureIBMODistance}.
  We devote Section \ref{sec:Sobolev} to the application of our methods, showing Theorem \ref{thm:sobolev}.
  Finally, in Section \ref{sec:OpenProblems} we state three open problems closely related to our results.

  It is a pleasure to thank Petros Galanopoulos, Oleg Ivrii and Martí Prats for several helpful conversations and interesting comments.
  
  \section{Preliminaries}
  \label{sec:Preliminaries}
  We need an auxiliary result that estimates the oscillation of the second divided differences when changing their centre and step size.
  For a continuous function $f$ we define its \emph{first divided difference} at $x \in \reals$ with step size $h > 0$ as
  \begin{equation*}
    \Delta_1f(x,h) = \frac{f(x+h)-f(x)}{h}.
  \end{equation*}
  For convenience, we may also denote $\Delta_1f(x,h) = \Delta_1f(I),$ where $I=(x,x+h).$
  
  \begin{lemma}
    \label{lemma:DividedDifferencesVariation}
    Let $f \in \zyg$ and assume that $h' > h > 0$ and $|x-t| < h'/2.$
    Then
    \begin{multline}
      \label{eq:DividedDifferencesVariation}
      |\Delta_2f(x,h)-\Delta_2f(t,h')| \lesssim \\
      \norm{f}{\ast} \left(\frac{h'-h}{h'} \left(1 + \log\frac{h'}{h'-h}\right) + \frac{|x-t|}{h'}\log\left( \frac{h'}{|x-t|}+1 \right) \right).
    \end{multline}
  \end{lemma}
  
  \begin{proof}
    We split the proof in two steps.
    First, we find an estimate for the case $h' = h$ and then another one for $x = t.$
    We start showing that, for $h > 0,$ when $|x-t| < h/2,$ then
    \begin{equation}
      \label{eq:VariationEqualStep}
      |\Delta_2f(x,h)-\Delta_2f(t,h)| \lesssim \norm{f}{\ast} \frac{|x-t|}{h} \log\left(\frac{h}{|x-t|} + 1\right).
    \end{equation}
    We claim that, if $|x-t| > h/2,$ then
    \begin{equation}
      \label{eq:FirstDividedDifferences}
      |\Delta_1f(x,h)-\Delta_1f(t,h)| \lesssim \norm{f}{\ast} \log\left(\frac{|x-t|}{h} + 1\right).
    \end{equation}
    Indeed, let $u$ be the harmonic extension of $f$ on the upper halfplane $\reals^2_{+}.$
    It is a well known fact (see Chapter V of \cite{ref:Stein} or \cite{ref:Llorente}) that
    \begin{equation*}
      \left| \frac{f(x+h)-f(x)}{h}-u_x(x,h) \right| \lesssim \norm{f}{\ast},
    \end{equation*}
    and that
    \begin{equation*}
      \sup_{(x,h) \in \reals^2_{+}} h |\nabla u_x(x,h)| \lesssim \norm{f}{\ast}.
    \end{equation*}
    Thus, if we denote by $\rho(a,b)$ the hyperbolic distance between two points $a,b \in \reals^2_{+},$ we get
    \begin{equation*}
      |u_x(x,h)-u_x(t,h)| \lesssim \norm{f}{\ast} \rho((x,h),(t,h)).
    \end{equation*}
    Using the estimate
    \begin{equation*}
      \rho((x,h),(t,h)) \lesssim \log\left(\frac{|x-t|}{h} + 1\right),
    \end{equation*}
    we get \eqref{eq:FirstDividedDifferences}.
    
    Now, assume $x > t$ without loss of generality, and $x-t < h/2.$
    Write
    \begin{align*}
      h(\Delta_2f(x,h)-\Delta_2f(t,h)) = &\left(f(x+h)-f(t+h)\right) - \left(f(x)-f(t)\right) \\
      + &\left(f(x-h)-f(t-h)\right) - \left(f(x)-f(t)\right)
    \end{align*}
    and apply \eqref{eq:FirstDividedDifferences} to the first two terms taking $x' = t+h,$ $t' = t$ and $h' = x-t,$ and to the last two taking $x' = t-h,$ $t' = t$ and $h' = x-t.$
    This shows \eqref{eq:VariationEqualStep}.
    
    Assume now that $h' > h > 0.$
    We want to see that
    \begin{equation}
      \label{eq:VariationEqualCentre}
      |\Delta_2f(x,h')-\Delta_2f(x,h)| \lesssim \norm{f}{\ast} \frac{h'-h}{h'} \left(1+\log\frac{h'}{h'-h}\right).
    \end{equation}
    First note the following identity
    \begin{multline*}
      \Delta_2f(x,h)-\Delta_2f(x,h') = \\
      \frac{h'-h}{h'} [\Delta_2f(x,h) - (\Delta_1f(x+h,h'-h) - \Delta_1f(x-h',h'-h)) ].
    \end{multline*}
    Using \eqref{eq:FirstDividedDifferences} on the last two terms, we get \eqref{eq:VariationEqualCentre}.
    Finally, \eqref{eq:DividedDifferencesVariation} is a direct consequence of \eqref{eq:VariationEqualStep} and \eqref{eq:VariationEqualCentre}. 
  \end{proof}
  
  \section{The Dyadic Results}
  \label{sec:Dyadic}
  A \emph{dyadic rational} is a number of the form $k2^{-n}$ with $k,n \in \integers.$
  For $n \geq 0,$ let $\mathcal{D}_n = \{[k2^{-n},(k+1)2^{-n})\colon k \in \integers\}.$
  For $n < 0,$ consider $m$ such that $n = -2m+1$ or $n = -2m,$ and let $t_n = (4^m-1)/3.$
  In this case, define $\mathcal{D}_n = \{[k2^{-n}-t_n,(k+1)2^{-n}-t_n)\colon k \in \integers\}.$
  A \emph{dyadic interval} $I$ is an interval such that $I \in \mathcal{D}_n$ for some $n \in \integers,$ and in this case we say that $I$ is a dyadic interval \emph{of generation $n$.}
  Denote by $\mathcal{D} = \bigcup_{n \in \integers} \mathcal{D}_n$ the set of all dyadic intervals.
  Note that, given $I \in \mathcal{D}_n$ for $n \in \integers,$ there is a unique interval $I^\ast$ in $\mathcal{D}_{n-1}$ that contains $I,$ which we call the \emph{predecessor of $I.$}
  If $I_0$ is an arbitrary interval, we will use the notation $\mathcal{D}(I_0) = \{I \in \mathcal{D} \colon I \subseteq I_0\}.$
  As explained in the introduction, a continuous real valued function $f$ on $\reals$ belongs to the \emph{dyadic Zygmund class,} denoted $f \in \zygdyc,$ if
  \begin{equation*}
    \norm{f}{\ast d} = \sup_{I \in \mathcal{D}} |\Delta_2f(I)| < \infty.
  \end{equation*}
  In a similar fashion, we say that a locally integrable function $f$ has \emph{bounded dyadic mean oscillation,} $f \in \BMO_d,$ if
  \begin{equation*}
    \norm{f}{\BMO d} = \sup_{I \in \mathcal{D}} \left(\frac{1}{|I|} \int_I |f(x)-f_I|^2\, dx\right)^{1/2} < \infty, 
  \end{equation*}
  and we consider the \emph{dyadic $\IBMO$} space to be the space of continuous real valued functions on $\reals$ whose distributional derivatives
  belong to $\BMO_d,$ that is
  \begin{equation*}
    \IBMO_d = \{f \in \continuous(\reals) \colon f' \in \BMO_d\}.
  \end{equation*}
  It is easy to see that each dyadic space contains its corresponding homogeneous space, that is $\BMO \subseteq \BMO_d$ and $\zyg \subseteq \zygdyc.$
  It is important to remark, as well, that none of these pairs are equal.
  More information on the relation between $\BMO$ and $\BMO_d$ can be found in \cite{ref:GarnettJonesBMOAndDyadicBMO}, \cite{ref:Mei} and \cite{ref:Conde}.
  
  The spaces $\zygdyc$ and $\IBMO_d$ can be regarded as well as spaces of dyadic martingales.
  We say that a sequence of functions $S = \{S_n\}$ is a \emph{dyadic martingale} if for all $n \geq 0$ the following conditions are satisfied:
  \begin{enumerate}[(i)]
    \item
    $S_n$ is constant on any $I \in \mathcal{D}_n,$
    \item
    $S_n|I = \frac{1}{2} \left(S_{n+1}|I^{(1)} + S_{n+1}|I^{(2)}\right)$ for all $I \in \mathcal{D}_n,$
    where $I^{(1)},I^{(2)}$ are the intervals in $\mathcal{D}_{n+1}$ contained in $I.$
  \end{enumerate}
  We will denote the value of $S_n$ at $I \in \mathcal{D}_m,$ $m \geq n,$ by $S_n(I),$ and, if there is no ambiguity, when $I \in \mathcal{D}_n$ we will just write $S(I).$
  For $x \in \reals$ and $n \geq 0,$ let $I \in \mathcal{D}_n$ be such that $x \in I.$
  Then, we have that $S_n(x) = S(I),$ and we will denote $S(x) = \lim_{n\to\infty} S_n(x)$ when this limit exists.
  For $n \geq 1,$ we call \emph{jump of $S$ at generation $n$} the function $\Delta S_n(x) = S_n(x) - S_{n-1}(x),$ and if $I \in \mathcal{D}_n,$ we use the notation $\Delta S_n(I) = S_n(I) - S_{n-1}(I^\ast),$ where $I^\ast$ is the predecessor of $I.$
  One can easily check that for a dyadic martingale $S$ the jumps $\Delta S_j$ and $\Delta S_k$ are orthogonal in $\Lp{2}(I)$ for any $I \in \mathcal{D}_0$ when $j \neq k.$
  
  With these concepts at hand, we can associate to each function $f \in \zygdyc$ a dyadic martingale $S,$ which we shall call \emph{the average growth martingale of $f,$} as follows.
  For a dyadic interval $I = [a,b) \in \mathcal{D}_n,$ set
  \begin{equation}
    \label{eq:AGMartingale}
    S_n(I) = \frac{f(b)-f(a)}{b-a} = 2^{n} (f(b)-f(a)).
  \end{equation}
  Now, observe that the second divided difference of $f$ can be expressed in terms of the jumps of $S;$ that is, for $I\in\mathcal{D}_n,$ we have the relation
  \begin{equation*}
    |\Delta_2f(I^{\ast})| = 2|\Delta S(I)|.
  \end{equation*}
  Now it is obvious that any dyadic martingale $S$ is related to a function $f \in \zygdyc$ (up to a linear term) through the relation \eqref{eq:AGMartingale} if and only if
  \begin{equation*}
    \norm{S}{\ast} = \sup_{I\in\mathcal{D}} |\Delta S(I)| < \infty.
  \end{equation*}
  To get the corresponding description of martingales associated with $\IBMO_d$ functions, we will discretise \eqref{eq:BMONorm}.
  Note that for $f \in \IBMO_d,$ with average growth martingale $S,$ and $I \in \mathcal{D}_N,$ using
  that the jumps $\{\Delta S_n\}_{n \geq N}$ restricted to $I$ of the martingale $S$ are orthogonal in $\Lp{2},$ one can express
  \begin{equation*}
    \int_I |f'(x) - f'_I|^2\, dx = \int_I \sum_{n > N} |\Delta S_n(x)|^2\, dx.
  \end{equation*}
  Thus, a martingale $S$ is related to a function $f \in \IBMO_d$ through the relation \eqref{eq:AGMartingale} if and only if
  \begin{equation}
    \label{eq:BMOMartingale}
    \norm{S}{\BMO} = \sup_{I\in\mathcal{D}} \left( \frac{1}{|I|} \sum_{J\in\mathcal{D}(I)} |\Delta S(J)|^2 |J| \right)^{1/2} < \infty.
  \end{equation}
  The analog of Theorem \ref{thm:IBMODistance} for this setting is the following
    
  \begin{thm}
    \label{thm:clibmod}
    Let $f$ be a compactly supported function in $\zygdyc.$
    For a fixed $\varepsilon > 0,$ define $D(f,\varepsilon)$ by
    \begin{equation}
      \label{eq:clibmodCondition}
      D(f,\varepsilon) = \sup_{I \in \mathcal{D}} \frac{1}{|I|} \sum_{\substack{J \in \mathcal{D}(I) \\ |\Delta_2f(J)| > \varepsilon}} |J|.
    \end{equation}
    Then,
    \begin{equation}
      \label{eq:clibmodDistance}
      \dist(f,\IBMO_d) = \inf \{\varepsilon > 0 \colon D(f,\varepsilon) < \infty\}.
    \end{equation}
  \end{thm}
  Note that we can rewrite this result in terms of martingales.
  Let $f\in\zygdyc$ be compactly supported on a dyadic interval $I_0,$ and consider its average growth martingale $S$ defined by \eqref{eq:AGMartingale}.
  In this way, $D(f,\varepsilon)$ in \eqref{eq:clibmodCondition} can be expressed as
  \begin{equation}
    \label{eq:clibmodMartingale}
    D(f,\varepsilon) = \sup_{I \in \mathcal{D}(I_0)} \frac{1}{|I|} \sum_{\substack{J \in \mathcal{D}(I) \\
    |\Delta S(J)| > \varepsilon/2}} |J|.
  \end{equation}
  
  \begin{proof}[Proof of Theorem \ref{thm:clibmod}]
    Without loss of generality, let us assume that $f$ is supported on the dyadic interval $I_0 = [0,1].$
    We need to prove that, for a given $\varepsilon > 0,$ there exists a function $b\in\IBMO_d$ satisfying
    $\norm{f-b}{\ast d} \leq \varepsilon$ if and only if $D(f,\varepsilon) < \infty.$
    Denote by $\varepsilon_0$ the infimum in the left-hand side of \eqref{eq:clibmodDistance}.
    
    Given $\varepsilon > \varepsilon_0,$ we will construct a function $b\in\IBMO_d$ such that $\norm{f-b}{\ast d} \leq \varepsilon.$
    Consider the average growth martingale $S$ for function $f,$ defined by \eqref{eq:AGMartingale}.
    First, we approximate the martingale $S$ by a martingale $B$ related to an $\IBMO_d$ function,
    that is satisfying \eqref{eq:BMOMartingale}.
    Take $B(I_0) = S(I_0)$ and construct $B$ by setting $\Delta B(J) = \Delta S(J)$ whenever $|\Delta S(J)| > \varepsilon/2$ and $\Delta B(J)=0$ otherwise, for $J\in\mathcal{D}(I_0).$
    
    By construction, it is clear that $\norm{S-B}{\ast} \leq \varepsilon/2.$
    Moreover, for any $I \in \mathcal{D},$ we have
    \begin{equation*}
      \sum_{J\in\mathcal{D}(I)} |\Delta B(J)|^2 |J| \leq \norm{S}{\ast}^2
      \sum_{\substack{J\in\mathcal{D}(I) \\ |\Delta S(J)|>\varepsilon/2}} |J| \leq |I| \norm{S}{\ast}^2 D(f,\varepsilon),
    \end{equation*}
    showing that $B$ satisfies \eqref{eq:BMOMartingale}.
    
    Now, using that the jumps $\Delta B_j$ and $\Delta B_k$ are orthogonal in $\Lp{2},$ we have
    \begin{equation*}
      \int_{I_0} \left( \sum_{n=1}^{\infty} \Delta B_n(x) \right)^2\, dx = \int_{I_0} \sum_{n=1}^{\infty} |\Delta B_n(x)|^2\, dx = 
      \sum_{J\in\mathcal{D}(I_0)} |\Delta B(J)|^2 |J| < \infty.
    \end{equation*}
    This gives that $\lim_{n\to\infty} B_n(x)$ exists at almost every point $x \in I_0$ and it is actually a square integrable function, so that we can integrate it to get $b(x) = \int_0^x \lim_n B_n(s)\, ds \in \IBMO_d$ such that $\norm{f-b}{\ast d} \leq \varepsilon.$
    
    Finally, if $\varepsilon < \varepsilon_0,$ we show that no function $b\in\IBMO_d$ satisfies $\norm{f-b}{\ast d} \leq \varepsilon.$
    Take $\varepsilon_0 > \varepsilon_1 > \varepsilon,$ assume that there is $b\in\IBMO_d$ satisfying
    $\norm{f-b}{\ast d} \leq \varepsilon,$ and let $S$ and $B$ be the respective average growth martingales for $f$ and $b.$
    For any $I\in\mathcal{D}$ such that $|\Delta S(I)| > \varepsilon_1,$ we have that $|\Delta B(I)| >
    \varepsilon_1 - \varepsilon = \delta > 0.$
    Thus
    \begin{equation*}
      \frac{1}{|I|} \sum_{J\in\mathcal{D}(I)} |\Delta B(J)|^2 |J| > 
      \frac{\delta^2}{|I|} \sum_{\substack{J\in\mathcal{D}(I) \\ |\Delta S(J)|>\varepsilon_1}} |J|.
    \end{equation*}
    The supremum of this quantity when $I$ ranges over all dyadic intervals is $\delta^2 D(f,\varepsilon_1) = +\infty.$
    This contradicts condition \eqref{eq:BMOMartingale} for martingale $B$ and, hence, that $b$ is an $\IBMO_d$ function, concluding the proof of the theorem.
  \end{proof}
  
  \section{From the Dyadic to the Continuous Setting}
  \label{sec:Continuous}
  Before proving Theorem \ref{thm:zygAverages}, let us make some observations.
  Consider the measurable mapping $\alpha\mapsto t^{(\alpha)}$ from $\reals$ to $\zygdyc$ such that all $t^{(\alpha)}$ are supported on $I_0 = [0,1]$ and such that $\norm{t^{(\alpha)}}{\ast d} \leq 1,$ and let $R \geq 1.$
  We will denote by $\mathcal{D}^0 = \mathcal{D}$ the standard dyadic filtration and by $\mathcal{D}^{\beta}$ the translated filtration by $-\beta$ units.
  We also extend this notation to denote by $\mathcal{D}^{0}_n$ the set of intervals of size $2^{-n}$ in $\mathcal{D}^0$ and by $\mathcal{D}^{\beta}_n$ the set of intervals of the same size in $\mathcal{D}^{\beta}.$
  Similarly, we denote by $\zyg^0$ the dyadic Zygmund class with respect to the filtration $\mathcal{D}^0$ and $\zyg^{\beta}$ the Zygmund dyadic class with respect to $\mathcal{D}^{\beta}.$
  With this notation, if $f(x)\in\zyg^0,$ then $f(x+\beta)\in\zyg^\beta.$
  
  Now consider an arbitrary interval $I$ and the adjacent interval $\tilde{I}=I-|I|$ of the same size. 
  Fix $R \geq 1$ and $\alpha\in[-R,R]$ and let $n$ be the minimum integer such that $I$ contains an interval of $\mathcal{D}^{\alpha}_n,$ and let $\mathcal{F}_n(I)$ be the set of all such intervals.
  For each $m>n,$ let $\mathcal{F}_m(I)$ be the set of intervals $J \in \mathcal{D}^{\alpha}_m$ such that $J \subset I \setminus \cup_{j=n}^{m-1} \mathcal{F}_j(I).$
  Then, $\mathcal{F}(I) = \cup_{j\geq n} \mathcal{F}_j(I)$ is a covering of $I$ by intervals of $\mathcal{D}^{\alpha}.$
  The covering $\mathcal{F}(\tilde{I})$ of $\tilde{I}$ is constructed in the exact same way.

  Let us say that $\mathcal{F}(I) = \{I_j\}_{j=1}^\infty.$
  We may assume that the intervals $I_j$ are ordered in the following way.
  Whenever $j>k,$ $|I_j| \leq |I_k|,$ and we may take $I_k$ to be to the left of $I_j$ if $|I_j|=|I_k|.$
  That is, we order the intervals decreasing in size and left to right for those that have the same length.
  We consider the covering $\mathcal{F}(\tilde{I}) = \{\tilde{I}_j\}_{j=1}^\infty$ to be ordered in the same way.

  \begin{lemma}
    \label{lemma:maximality}
    Let $I \subseteq \reals$ be a finite interval and $\mathcal{F}(I)$ its covering by intervals of $\mathcal{D}^\alpha$ constructed and ordered as previously explained.
    Then, the intervals of $\mathcal{F}(I)$ have disjoint interiors and, for $j \geq 1,$ they satisfy that $|I_{j+2}| \leq |I_j|/2.$
  \end{lemma}
  \begin{proof}
    The intervals in $\mathcal{F}(I)$ have disjoint interiors by construction.
    Moreover, these intervals are maximal in the sense that if $J \in \mathcal{F}(I)$ and $J \subsetneq J' \in \mathcal{D}^{\alpha},$ then $J' \not\subset I.$
    Thus, it is clear that for each $n \geq 1$ there are at most two intervals in $\mathcal{F}(I)$ of size $2^{-n}|I|.$
    This yields that, for $j \geq 1,$ the intervals in $\mathcal{F}(I)$ satisfy $|I_{j+2}| \leq |I_j|/2.$
  \end{proof}
  
  When $|I| = 2^{-n}$ for some $n \in \integers,$ the covering $\mathcal{F}(\tilde{I}) = \{\tilde{I}_j\}_{j=1}^\infty$ is a translation of $\mathcal{F}(I) = \{I_j\}_{j=1}^\infty.$
  More precisely, if we order both $\{\tilde{I}_j\}$ and $\{I_j\}$ as previously explained, then for each $j \geq 1$ we have that $\tilde{I}_j = I_j - |I|$ and, trivially, for every $j \geq 1,$ $|\tilde{I}_j| = |I_j|.$
  However, for an arbitrary interval $I,$ the sizes of the intervals in $\mathcal{F}(I)$ and $\mathcal{F}(\tilde{I})$ may be completely different.
  For instance, it could happen that for a given $j \in \integers,$ $\mathcal{F}(I)$ had two intervals of
  size $2^{-j}$ while $\mathcal{F}(\tilde{I})$ had only one.

  \begin{lemma}
    \label{lemma:EqualLengths}
    Let $I$ and $\tilde{I}$ be two adjacent intervals of the same length.
    Fix $\alpha \in \reals.$
    Then there are coverings $\mathcal{G}(I) = \{J_j\}$ and $\mathcal{G}(\tilde{I}) = \{\tilde{J}_j\},$ of $I$ and $\tilde{I}$ respectively, both consisting of intervals of $\mathcal{D}^\alpha,$ with $|J_j| = |\tilde{J}_j|$ for any $j,$ and with $|J_{j+2}| \leq |J_j|/2.$
  \end{lemma}

  \begin{proof}
    Consider the previous coverings $\mathcal{F}(I) = \{I_j\}$ and $\mathcal{F}(\tilde{I}) = \{\tilde{I}_j\}.$
    If $|I_1| = |\tilde{I}_1|,$ then take $J_1 = I_1$ and $\tilde{J}_1 = \tilde{I}_1.$
    If these sizes are different, assume $|I_1| > |\tilde{I}_1|$ (otherwise the procedure is the same), there exists an integer $k \geq 2$ such that $\sum_{j=1}^k |\tilde{I}_j| = |I_1|.$
    Note that $k$ exists because all $|I_j|$ (and also $|\tilde{I_j}|)$ are dyadic rationals that add up to $|I| = |\tilde{I}|.$
    Then take $\tilde{J}_j = \tilde{I}_j$ for $1 \leq j \leq k,$ and choose pairwise disjoint intervals $J_1,\ldots,J_k \in \mathcal{D}(I_1)$ such that $I_1 = \cup_{l=1}^k J_l$ and that, for each $1 \leq j \leq k,$ $|J_j| = |\tilde{J}_j|.$
    Note that, for $1 \leq j \leq k-2,$ we have that $|J_{j+2}| = |\tilde{J}_{j+2}| \leq |\tilde{J}_j|$ because of Lemma \ref{lemma:maximality}.
    Recursively, consider that we have fixed $\{J_j\}_{j=1}^p \in \mathcal{G}(I)$ and $\{\tilde{J}_j\}_{j=1}^p \in \mathcal{G}(\tilde{I}),$ let $m,n$ be the smallest integers such that $I_m \subseteq I \setminus \cup_{j=1}^p J_j$ and $\tilde{I}_n \subseteq \tilde{I} \setminus \cup_{j=1}^p \tilde{J}_j,$ and repeat the previous step with $I_m$ and $\tilde{I}_n.$
  \end{proof}
  
  Given two finite intervals $I_1,I_2,$ we say that their \emph{minimal common predecessor in $\mathcal{D}^{\alpha},$} denoted $P_\alpha(I_1,I_2),$ is the interval $P_\alpha(I_1,I_2)\in\mathcal{D}^{\alpha}$ such that $I_1\cup I_2 \subseteq P_\alpha(I_1,I_2)$ and such that for every $J \in \mathcal{D}^{\alpha}$ that satisfies $I_1\cup I_2 \subseteq J,$ then $P_\alpha(I_1,I_2) \subseteq J.$
  If $I_1,I_2 \in \mathcal{D}^{\alpha},$ we define their \emph{distance in the dyadic filtration $\mathcal{D}^{\alpha},$} denoted by $\dist_\alpha (I_1,I_2),$ as
  \begin{equation*}
    \dist_\alpha (I_1,I_2) = \log_2\frac{|P_\alpha(I_1,I_2)|}{|I_1|} + \log_2\frac{|P_\alpha(I_1,I_2)|}{|I_2|}.
  \end{equation*}
  Here it is necessary to specify the index $\alpha$ as one could have two intervals $I_1,I_2$ that were dyadic in two different filtrations $\mathcal{D}^\alpha$ and $\mathcal{D}^\beta$ such that the difference between both distances is as large as desired.
  
  \begin{lemma}
    \label{lemma:DividedDifferencesDyadicDistance}
    Consider $f \in \zygdyc^\alpha$ and $I,J \in \mathcal{D}^\alpha.$
    Then,
    \begin{equation*}
      |\Delta_1f(I)-\Delta_1f(J)| \leq \norm{f}{\ast\, d} \dist_\alpha(I,J).
    \end{equation*}
  \end{lemma}
  \begin{proof}
    Consider the sequences $\{I_j\}_{j=0}^k$ and $\{J_j\}_{j=0}^l$ in $\mathcal{D}^\alpha$ such that $I_0 = I,$ $J_0 = J,$ $I_k = J_l = P_\alpha(I,J),$ and such that $I_j^\ast = I_{j+1}$ for $0 \leq j < k,$ and such that $J_j^\ast = J_{j+1}$ for $0 \leq j < l.$
    One has that
    \begin{equation*}
      |\Delta_1f(I)-\Delta_1f(J)| \leq \sum_{j=0}^{k-1} |\Delta_1f(I_j)-\Delta_1f(I_{j+1})| + \sum_{j=0}^{l-1} |\Delta_1f(J_j)-\Delta_1f(J_{j+1})|.
    \end{equation*}
    Each term of these sums is bounded by $\norm{f}{\ast}$ and, since there are exactly $\dist_\alpha(I,J)$ terms, the result follows.
  \end{proof}

  For future convenience, given a finite interval $I,$ we will denote its midpoint by $c(I).$

  \begin{lemma}
    \label{lemma:CommonPredecessorSize}
    Fix $R \geq 1$ and let $I$ and $\tilde{I}$ be two adjacent intervals of the same length.
    Let $N$ be the integer such that $2^{-N-1} < |I| \leq 2^{-N}$ and let $M$ be the integer such that $2^{M-1} < R \leq 2^M.$ Then, for each $k \geq 1,$ one has that
    \begin{equation*}
      |\{\alpha \in [-R,R]\colon |P_\alpha(I,\tilde{I})| = 2^{k-N}\}| \leq 2^{M+1}2^{-k+2}.
    \end{equation*}
  \end{lemma}
  \begin{proof}
    Note that for any value of $\alpha,$ one has that $|P_\alpha(I,\tilde{I})|=2^k2^{-N}$ for some positive integer $k.$
    For $k \geq 2,$ the size of the minimal common predecessor in $\mathcal{D}^\alpha$ is exactly $2^{k-N}$ if and only if there is some $J \in \mathcal{D}^\alpha,$ with $|J|=2^{k-N},$ such that $c(J) \in I\cup\tilde{I}.$
    For the case $k=1,$ it is only true that if $J \in \mathcal{D}^{\alpha},$ with $|J| = 2^{1-N},$ is the minimal common predecessor, then $c(J) \in I\cup\tilde{I},$ while the reciprocal does not hold.

    Consider $J \in \mathcal{D}_{N-k},$ and consider as well the translated intervals $J+\alpha,$ for $\alpha \in [-R,R],$ and their midpoints $c(J+\alpha).$
    The set $\{\alpha \in [-R,R] \colon c(J+\alpha) \in I\cup\tilde{I}\}$ has measure $2^{-N+1},$ which is the measure of $I\cup\tilde{I}.$
    Note that it is actually here that we implicitly use that $R \geq 1,$ since otherwise it could be that this set had length $2^{M+1}.$
    Since in $[-R,R]$ there are at most $2^{M+1}2^{N-k+1}$ intervals of length $2^{k-N},$ the result follows immediately.
  \end{proof}
  
  \begin{proof}[Proof of Theorem \ref{thm:zygAverages}]
    Assume without loss of generality that $I_0 = [0,1]$ and fix $R \geq 1.$
    We just need to check that
    \begin{equation*}
      \sup_I |\Delta_1t_R(I)-\Delta_1t_R(\tilde{I})| \leq C < \infty,
    \end{equation*}
    where $I$ ranges over all finite intervals, with $\tilde{I} = I - |I|,$ and where $C$ is independent of the value of $R.$
    Fix a finite interval $I$ and consider the integer $N$ such that $2^{-N-1} < |I| \leq 2^{-N}.$ First, we express
    \begin{equation*}
      \Delta_1t_R(I)-\Delta_1t_R(\tilde{I}) = \frac{1}{2R} \int_{-R}^R \left( \Delta_1t^{(\alpha)}(\alpha+I)-\Delta_1t^{(\alpha)}(\alpha
      +\tilde{I}) \right)\, d\alpha.
    \end{equation*}
    Now, for a given $\alpha,$ consider the coverings $\mathcal{G}^\alpha(I) = \{I_j\}_{j=1}^\infty$ and $\mathcal{G}^\alpha(\tilde{I}) = \{\tilde{I}_j\}_{j=1}^\infty$ given in Lemma \ref{lemma:EqualLengths}, that satisfy $|I_j| = |\tilde{I}_j|$ for $j \geq 1.$
    We can express
    \begin{equation*}
      |\Delta_1t^{(\alpha)}(\alpha+I)-\Delta_1t^{(\alpha)}(\alpha+\tilde{I})| \leq \sum_{j \geq 1} \frac{|I_j|}{|I|}
      \left|\Delta_1t^{(\alpha)}(\alpha + I_j) - \Delta_1t^{(\alpha)}(\alpha + \tilde{I}_j)\right|
    \end{equation*}
    Observe that $\alpha + I_j \in \mathcal{D}^0$ and, since $t^{(\alpha)} \in \zygdyc^0,$ using Lemma \ref{lemma:maximality}
    and Lemma \ref{lemma:DividedDifferencesDyadicDistance} we may bound the previous quantity by
    \begin{equation*}
      \sum_{j \geq 1} \frac{|I_j|}{|I|} \norm{t^{(\alpha)}}{\ast d} \dist_\alpha(I_j,\tilde{I}_j)
      \lesssim \sum_{j \geq 1} 2^{-\frac{j}{2}} \log \left( 2^{N + \frac{j}{2}} |P_\alpha(I,\tilde{I})| \right),
    \end{equation*}
    where we have also used that $\norm{t^{(\alpha)}}{\ast d} \leq 1$ for every $\alpha.$ Summing over $j,$ we get
    \begin{equation*}
      |\Delta_1t^{(\alpha)}(\alpha+I)-\Delta_1t^{(\alpha)}(\alpha+\tilde{I})| \lesssim 1 + N + \log |P_\alpha(I,\tilde{I})|.
    \end{equation*}
    Averaging over $\alpha,$ we have
    \begin{equation*}
      |\Delta_1t_R(I)-\Delta_1t_R(\tilde{I})| \lesssim \frac{1}{R} \int_{-R}^R \left( 1 + N + \log |P_\alpha(I,\tilde{I})| \right)\, d\alpha.
    \end{equation*}
    Set $R_k = \{\alpha \in [-R,R] \colon |P_\alpha(I,\tilde{I})|=2^{k-N}\}$ and recall that, by Lemma
    \ref{lemma:CommonPredecessorSize}, $|R_k| \leq 2^{M+1}2^{-k+2},$ where $M$ is the integer such that $2^{M-1} < R \leq 2^M.$ Then, we can bound the last quantity by
    \begin{equation*}
      2^{-M+1} \sum_{k \geq 1} \int_{R_k} \left( 1 + N + \log |P_\alpha(I,\tilde{I})| \right)\, d\alpha \lesssim \sum_{k \geq 1} 2^{-k} \left( 1 + N + \log 2^{k-N} \right),
    \end{equation*}
    which is bounded by some positive constant $C.$
    Note that the factors depending on $N$ and on $M$ cancel out, which means that this last constant depends neither on $R$ nor on $I.$
  \end{proof}
  
  We are now ready to prove Theorem \ref{thm:IBMODistance}.
  
  \begin{proof}[Proof of Theorem \ref{thm:IBMODistance}]
    Let $f \in \zyg$ and let $\varepsilon_0$ be the infimum in the left-hand side of \eqref{eq:IBMODistance}. First we show that, whenever $\varepsilon < \varepsilon_0,$ there is no function $b \in \IBMO$ such that $\norm{f-b}{\ast} \leq \varepsilon.$
    Indeed, assume that for a given $\varepsilon < \varepsilon_0$ there is $b \in \IBMO$ such that $\norm{f-b}{\ast} \leq \varepsilon.$
    Take $\varepsilon < \varepsilon_1 < \varepsilon_0$ and note that, whenever $|\Delta_2f(x,h)| > \varepsilon_1$ we have that $|\Delta_2b(x,h)| > \varepsilon_1-\varepsilon = \delta > 0.$
    In particular, this means that $A(f,\varepsilon_1) \subseteq A(b,\delta).$
    Thus,
    \begin{equation*}
      \frac{1}{|I|} \int_I\int_0^{|I|} |\Delta_2b(x,h)|^2 \frac{dh\, dx}{h} \geq \frac{\delta^2}{|I|} \int_I\int_0^{|I|} \chi_{A(f,\varepsilon_1)}(x,h) \frac{dh\, dx}{h},
    \end{equation*}
    but the supremum, with $I$ ranging over all finite intervals, of the later quantity is not finite since $\varepsilon_1 < \varepsilon_0.$
    By Theorem \ref{thm:Strichartz}, this contradicts that $b \in \IBMO.$
      
    We are left with showing that there exists a universal constant $C > 0$ such that, for any $\varepsilon > \varepsilon_0,$ there is $b=b(\varepsilon) \in \IBMO$ such that $\norm{f-b}{\ast} \leq C\varepsilon.$
    For any such $\varepsilon,$ by assumption we have that
    \begin{equation}
      \label{eq:ContinuousSupremum}
      C(f,\varepsilon) = \sup_{I} \frac{1}{|I|} \int_I\int_0^{|I|} \chi_{A(f,\varepsilon)}(x,h) \frac{dh\, dx}{h} < \infty.
    \end{equation}
    Assume now, without loss of generality, that $f$ has support in $I_0 = [0,1].$
    We claim that \eqref{eq:ContinuousSupremum} implies that $D(f,\varepsilon)$ defined by \eqref{eq:clibmodMartingale} is finite.
    To see this, take $\varepsilon > \varepsilon_1 > \varepsilon_0,$ and let $J \in \mathcal{D}$ be such that $|\Delta S(J)| > \varepsilon/2,$ which is equivalent to say that $|\Delta_2f(c(J^\ast),|J|)| > \varepsilon.$
    By Lemma \ref{lemma:DividedDifferencesVariation}, there exists $\delta > 0$ such that if $|x - c(J^\ast)| < \delta |J|$ and $1-\delta < h/|J| < 1+\delta,$ then $|\Delta_2f(x,h)| > \varepsilon_1.$
    Applying this to every dyadic interval $J$ with $|\Delta S(J)| > \varepsilon/2,$ we find the upper bound
    \begin{equation*}
      \frac{1}{|I|} \sum_{\substack{J \in \mathcal{D}(I) \\ |\Delta S(J)| > \varepsilon/2}} |J| \lesssim \frac{1}{|I|} \int_I\int_0^{|I|} \chi_{A(f,\varepsilon_1)}(x,h) \frac{dh\, dx}{h} \leq C(f,\varepsilon_1)
    \end{equation*}
    for all $I \in \mathcal{D}.$
    Thus,
    \begin{equation}
    \label{eq:DiscreteDistanceEstimate}
      D(f,\varepsilon) = \sup_{I \in \mathcal{D}} \frac{1}{|I|} \sum_{\substack{J \in \mathcal{D}(I) \\ |\Delta S(J)| > \varepsilon/2}} |J| \lesssim C(f,\varepsilon_1).
    \end{equation}
    
    Next, for each $\alpha\in[-1,1],$ define $f^{(\alpha)}(x) = f(x-\alpha) \in \zygdyc.$
    By \eqref{eq:DiscreteDistanceEstimate} and Theorem \ref{thm:clibmod}, $\dist(f^{(\alpha)},\IBMO_d) \leq \varepsilon.$
    Hence, there are $b^{(\alpha)} \in \IBMO_d$ and $t^{(\alpha)} \in \zygdyc$ such that $f^{(\alpha)} = b^{(\alpha)} + t^{(\alpha)},$ with $\norm{t^{(\alpha)}}{\ast d} \leq \varepsilon$ for all $\alpha \in [-1,1].$
    This allows us to express
    \begin{align*}
      f(x) &= \frac{1}{2} \int_{-1}^1 f^{(\alpha)}(x+\alpha)\, d\alpha \\
      &= \frac{1}{2} \int_{-1}^1 b^{(\alpha)}(x+\alpha)\, d\alpha + \frac{1}{2} \int_{-1}^1 t^{(\alpha)}(x+\alpha)\, d\alpha.
    \end{align*}
    By Theorem \ref{thm:DyadicBMOtoBMO}, taking $R = 1,$ the first integral yields a function $b \in \IBMO.$
    By Theorem \ref{thm:zygAverages}, with $R = 1$ as well, the second integral yields a function $t \in \zyg$ with $\norm{t}{\ast} \leq C \varepsilon,$ where the later constant is the same that appears in Theorem \ref{thm:zygAverages}.
    This completes the proof.
  \end{proof}
  
  \section{The Higher Dimensional Result}
  \label{sec:Measures}
  For a Borel set $A \subset \reals^d,$ denote by $|A|$ its Lebesgue measure.
  Let $x = (x_1,\ldots,x_d) \in \reals^d$ and $h > 0$ and denote by $Q(x,h)$ the cube centred at $x$ and with side length $l(Q) = h$ and edges parallel to the axis.
  For a signed Borel measure $\mu$ on $\reals^d,$ we will treat its densities on cubes as \emph{first divided differences,} and denote them by
  \begin{equation*}
    \Delta_1\mu(x,h) = \frac{\mu(Q(x,h))}{|Q(x,h)|}, \quad (x,h) \in \reals^{d+1}_+,
  \end{equation*}
  and we also define its \emph{second divided differences on cubes} as
  \begin{equation*}
    \Delta_2\mu(x,h) = \Delta_1\mu(x,h) - \Delta_1\mu(x,2h), \quad (x,h) \in \reals^{d+1}_+.
  \end{equation*}
  Here, $\reals^{d+1}_+$ denotes the upper halfspace, $\reals^{d+1}_+ = \{(x,h)\colon x \in \reals^d,\, h > 0\}.$
  We say that a signed Borel measure $\mu$ on $\reals^d$ is a Zygmund measure, $\mu \in \zyg,$ if it satisfies
  \begin{equation*}
    \norm{\mu}{\ast} = \sup_{(x,h) \in \reals^{d+1}_+} |\Delta_2\mu(x,h)| < \infty.
  \end{equation*}
  Note that there can be Zygmund measures that are singular with respect to the Lebesgue measure (see \cite{ref:Kahane} and \cite{ref:AleksandrovAndersonNicolau}).
  Recall that a real valued function $f$ on $\reals^d$ is said to have \emph{bounded mean oscillation in $\reals^d,$} $f \in \BMO(\reals^d),$ if
  \begin{equation*}
    \norm{f}{\BMO} = \sup_Q \left(\frac{1}{|Q|} \int_Q |f(x)-f_Q|^2\, dx\right)^{1/2} < \infty,
  \end{equation*}
  where $Q$ ranges over all finite cubes in $\reals^d$ with edges parallel to the axis and $f_Q = \frac{1}{|Q|} \int_Q f(x)\, dx.$
  We will say that a signed Borel measure $\nu$ on $\reals^d$ is an $\IBMO$ measure, $\nu \in \IBMO,$ if it is absolutely continuous with respect to the Lebesgue measure and its Radon-Nikodym derivative is
  \begin{equation*}
    d\nu (x) = b(x) dx
  \end{equation*}
  for some function $b \in \BMO(\reals^d).$
  Using a characterisation of $\BMO(\reals^d)$ functions due to R.~Strichartz (see \cite{ref:Strichartz}), one can see that such a measure $\nu$ satisfies
  \begin{equation}
    \label{eq:IBMOMeasureCharacterisation}
    \sup_Q \left(\frac{1}{|Q|} \int_Q \int_0^{l(Q)} |\Delta_2\nu(x,h)|^2 \frac{dh\, dx}{h}\right)^{1/2} < \infty.
  \end{equation}
  Conversely, whenever $\nu$ satisfies equation \eqref{eq:IBMOMeasureCharacterisation}, it is an absolutely continuous measure with Radon-Nikodym derivative in $\BMO(\reals^d)$ (see \cite{ref:DoubtsovNicolau}).

  Here we state a version of Theorem \ref{thm:IBMODistance} for Zygmund measures in $\reals^d.$
  For a given Zygmund measure $\mu$ and $\varepsilon > 0,$ consider the set
  \begin{equation*}
    A(\mu,\varepsilon) = \left\{ (x,h) \in \reals^{d+1}_+ \colon |\Delta_2\mu (x,h)| > \varepsilon \right\}.
  \end{equation*}
  \newtheorem*{thm:MeasureIBMODistance}{Theorem \ref{thm:MeasureIBMODistance}}
  \begin{thm:MeasureIBMODistance}
    Let $\mu$ be a compactly supported Zygmund measure on $\reals^d.$
    For each $\varepsilon > 0$ consider
    \begin{equation*}
      C(\mu,\varepsilon) = \sup_Q \frac{1}{|Q|} \int_Q\int_0^{l(Q)} \chi_{A(\mu,\varepsilon)}(x,h) \frac{dh\, dx}{h},
    \end{equation*}
    where $Q$ ranges over all finite cubes with edges parallel to the axis.
    Then
    \begin{equation*}
      \dist(\mu,\IBMO) \simeq \inf \{ \varepsilon > 0 \colon C(\mu,\varepsilon) < \infty \}.
    \end{equation*}
  \end{thm:MeasureIBMODistance}

  The proof of this result follows the same lines as that of Theorem \ref{thm:IBMODistance}.
  Nonetheless, one has to adapt the auxiliary results used in showing that theorem.
  First, we state and show the technical estimate in $\reals^d$ which is analogous to Lemma \ref{lemma:DividedDifferencesVariation}.
  For convenience, given $y \in \reals^{d-1}$ and $h > 0$ we will denote by $q(y,h)$ the cube in $\reals^{d-1}$ centred at $y,$ with sidelength $l(Q) = h$ and edges parallel to the axis.
  
  \begin{lemma}
    Let $\mu \in \zyg$ and assume that $h' > h > 0$ and $|x-t| < h/2.$
    Then
    \begin{multline*}
      |\Delta_2\mu(x,h) - \Delta_2\mu(t,h')|  \\
      \leq C_d\norm{\mu}{\ast} \left( \frac{h'-h}{h}\left(1+\log\left(\frac{h}{h'-h}+1\right)\right) + \frac{|x-t|}{h}\log\left(\frac{h}{|x-t|}+1\right) \right).
    \end{multline*}
    Here, the constant $C_d$ does only depend on the dimension $d.$
  \end{lemma}
  \begin{proof}
    The proof is split in two steps.
    First, we find an estimate for the case $h=h'$ and then another one for $x=t.$
    We start showing that, for $h > 0,$ when $|x-t| < h/2$
    \begin{equation}
      \label{eq:MeasureDifferencesTrans}
      |\Delta_2\mu(x,h)-\Delta_2\mu(t,h)| \leq C_d \norm{\mu}{\ast} \frac{|x-t|}{h} \log\left(\frac{h}{|x-t|}+1\right).
    \end{equation}
    First, if $|x-t| > h/2,$ then
    \begin{equation}
      \label{eq:MeasureFirstDifferences}
      |\Delta_1\mu(x,h)-\Delta_1\mu(t,h)| \lesssim \norm{\mu}{\ast}\log\left(\frac{|x-t|}{h}+1\right).
    \end{equation}
    The argument to show this bound is the same as in Lemma \ref{lemma:DividedDifferencesVariation}.
    The only difference is that one has to consider $u$ to be the harmonic extension of $\mu$ on the upper half-space $\reals^{d+1}_+,$ which will be itself a Bloch function, and use the well known fact (see Chapter V of \cite{ref:Stein} or \cite{ref:Llorente}) that
    \begin{equation*}
      |\Delta_1\mu(x,h)-u(x,h)| \lesssim \norm{\mu}{\ast}.
    \end{equation*}
    To show \eqref{eq:MeasureDifferencesTrans}, assume without loss of generality that $x = (x_1,\ldots,x_{d-1},x_d)$ and $t = (x_1,\ldots,x_{d-1},x_d')$ with $x_d' < x_d,$ and $|x-t| < h/2.$
    If we denote $y = (x_1,\ldots,x_{d-1}) \in \reals^{d-1},$ one can see that
    \begin{equation*}
      |Q(x,h)| \left(\Delta_2\mu(x,h) - \Delta_2\mu(t,h)\right) =
      \mu(l_+) - \mu(l_-) + \frac{\mu(L_-)}{2^d} - \frac{\mu(L_+)}{2^d},
    \end{equation*}
    where $l_+ = q(y,h) \times [x_d'+h/2,x_d+h/2),$ $l_- = q(y,h) \times [x_d'-h/2,x_d-h/2),$  $L_+ = q(y,2h) \times [x_d'+h,x_d+h)$ and $L_- = q(y,2h) \times [x_d'-h,x_d-h)$ are parallelepipeds at opposite sides of the cubes $Q(x,h)$ and $Q(x,2h)$ respectively.
    We just show how to estimate $|\mu(l_+)-\mu(l_-)|,$ as the other terms are estimated in the same way.
    The idea here is to cover $l_+$ with cubes $\{P_j\}$ and to use a translated cover $\{R_j\}$ for $l_-.$
    
    \begin{figure}[ht]
      \begin{center}
        \input{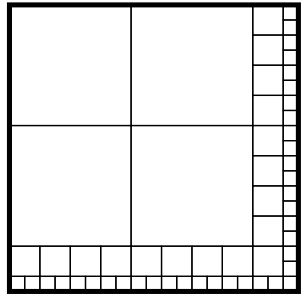}
        \caption{Parallelepiped $l_+$ seen from its base (bold square), and the distribution of the cubes $P_j^m.$ Cubes of the same size belong to the same generation.}
        \label{fig:CubesDistributionParallelepiped}
      \end{center}
    \end{figure}
    
    In order to cover $l_+$ with the appropriate cubes, write first $\frac{h}{|x-t|} = \sum_{n\geq 0} k_n2^{-n},$ where $k_0 \geq 2$ and $k_n$ is $0$ or $1$ for $n\geq 1$ (as in a binary expansion).
    We construct a generation $0$ placing $k_0^{d-1}$ cubes with mutually disjoint interiors of side length $|x-t|$ at one of the corners of $l_-,$ forming alltogether a smaller parallelepiped with one side of length $|x-t|$ and the rest of length $k_0|x-t|.$
    Let us denote by $\{P_0^{m}\}$ the set of cubes of generation $0.$
    Assume we have constructed cubes up to generation $j-1,$ that is, we have chosen $\{P_i^m\}_{i=0}^{j-1}.$
    At generation $j$ either we do nothing if $k_j = 0$ or, when $k_j = 1,$ we add a layer of cubes $\{P_j^m\}$ of side length $2^{-j}|x-t|,$ such that $\{P_i^m\}_{i=0}^j$ have pairwise disjoint interiors, in order to get a new square based parallelepiped with one side length $|x-t|$ and the rest of $(\sum_{n=0}^j k_n2^{-n})|x-t|$ (see Figure \ref{fig:CubesDistributionParallelepiped}).
    Let $\{P_j^m\}$ be the cubes of generation $j$ and note that their total volume is
    \begin{equation*}
      \sum_m |P_j^m| = |x-t|^d \left[\left(\sum_{n=0}^jk_n2^{-n}\right)^{d-1} - \left(\sum_{n=0}^{j-1}k_n2^{-n}\right)^{d-1}\right].
    \end{equation*}
    If $d \geq 2$ we deduce
    \begin{equation*}
      \sum_m |P_j^m| \lesssim d |x-t|^d k_j2^{-j} \left(\sum_{n = 0}^j k_n2^{-n}\right)^{d-2}.
    \end{equation*}
    Since $\sum_{n=0}^j k_n 2^{-n} \leq h/|x-t|,$ we deduce that
    \begin{equation}
      \label{eq:LayerVolumeBound}
      \sum_m |P_j^m| \lesssim d |x-t|^2 h^{d-2} k_j 2^{-j}
    \end{equation}
    Since the distance between the centres of $P_j^m$ and $R_j^m$ is bounded by a fixed multiple of $h,$ applying equation \eqref{eq:MeasureFirstDifferences} we get that
    \begin{equation}
      \label{eq:CubesDifferences}
      |\mu(P_j^m)-\mu(R_j^m)| \lesssim \norm{\mu}{\ast} |P_j^m| \log\left(\frac{h}{l(P_j^m)}+1\right),
    \end{equation}
    and using \eqref{eq:CubesDifferences}, we have
    \begin{align*}
      |\mu(l_+)-\mu(l_-)| &\leq \sum_j\sum_m |\mu(P_j^m)-\mu(R_j^m)| \\
      &\leq C \norm{\mu}{\ast} \sum_j\sum_m |P_j^m| \log\left(\frac{h}{2^{-j}|x-t|}+1\right).
    \end{align*}
    Summing over $m$ and using \eqref{eq:LayerVolumeBound}, this is bounded by
    \begin{equation*}
      C\norm{\mu}{\ast} d |x-t|^2 h^{d-2} \log\left(\frac{h}{|x-t|}+1\right) \sum_jk_j2^{-j}j,
    \end{equation*}
    and we deduce that
    \begin{equation}
      \label{eq:DifferenceParallelepiped}
      |\mu(l_+) - \mu(l_-)| \leq C d\norm{\mu}{\ast} |x-t| h^{d-1} \log\left(\frac{h}{|x-t|}+1\right).
    \end{equation}
    This and the analog estimate for $|\mu(L_+)-\mu(L_-)|$ yield estimate \eqref{eq:MeasureDifferencesTrans}.

    The second step is to show that, if $h' > h > 0,$ then
    \begin{equation}
      \label{eq:MeasureDifferencesDilate}
      |\Delta_2\mu(x,h')-\Delta_2\mu(x,h)| \leq C_d \norm{\mu}{\ast} \frac{h'-h}{h} \left(1+\log\left(\frac{h}{h'-h}+1\right)\right).
    \end{equation}
    Let $R(x,h,h') = q(y,h) \times [x_d-h'/2,x_d+h'/2),$ where $y \in \reals^{d-1}$ is such that $x = (y,x_d).$
    Note that $R(x,h,h')$ is the parallelepiped obtained from dilating the cube $Q(x,h)$ just in one direction.
    Denote as well
    \begin{equation*}
      \Delta_2\mu(x,h,h') = \frac{\mu(R(x,h,h'))}{|R(x,h,h')|} - \frac{\mu(R(x,2h,2h'))}{|R(x,2h,2h')|}.
    \end{equation*}
    To show \eqref{eq:MeasureDifferencesDilate}, it is enough to see that
    \begin{equation}
      \label{eq:SecondDividedDifferenceParallelepiped}
      |\Delta_2\mu(x,h,h')-\Delta_2\mu(x,h)|\! \leq\! C_d \norm{\mu}{\ast} \frac{h'-h}{h} \left(1+\log\left(\frac{h}{h'-h}+1\right)\right).
    \end{equation}
    
    \begin{figure}[ht]
      \begin{center}
        \input{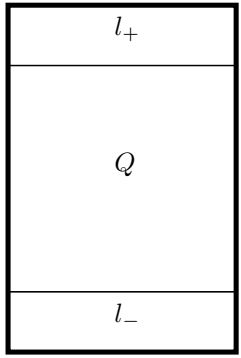}
        \caption{The parallelepiped $R$ can be decomposed into the cube $Q$ and the square based parallelepipeds $l_+$ and $l_-.$}
        \label{fig:RDecomposition}
      \end{center}
    \end{figure}
    
    Let us denote $Q=Q(x,h),$ $\tilde{Q}=Q(x,2h),$ $R=R(x,h,h')$ and $\tilde{R}=R(x,2h,2h').$
    Note that we can decompose $R$ as the disjoint union $Q \cup l_+ \cup l_-,$ where $l_+$ and $l_-$ are parallelepipeds similar to the ones we used before (see Figure \ref{fig:RDecomposition}).
    In the same way, decompose $\tilde{R} = \tilde{Q} \cup L_+ \cup L_-,$ and note that $L_+$ (and also $L_-)$ can be regarded as the union $\bigcup_{i=1}^{2^d} L_+^i,$ where each $L_+^i$ is a translation of $l_+.$
    Now, express
    \begin{multline*}
      \Delta_2\mu(x,h) - \Delta_2\mu(x,h,h') = \frac{\mu(Q)}{|Q|} - \frac{\mu(\tilde{Q})}{|\tilde{Q}|} - \frac{\mu(R)}{|R|} + \frac{\mu(\tilde{R})}{|\tilde{R}|} = \\
      \frac{h'-h}{h'} \left(\frac{\mu(Q)}{|Q|} - \frac{\mu(\tilde{Q})}{|\tilde{Q}|}\right) - \left(\frac{\mu(l_+)}{h^{d-1}h'} - \frac{\mu(L_+)}{2^dh^{d-1}h'}\right) - \left(\frac{\mu(l_-)}{h^{d-1}h'} - \frac{\mu(L_-)}{2^dh^{d-1}h'}\right).
    \end{multline*}
    The first term is $\Delta_2\mu(x,h) (h'-h)/h',$ which is bounded by $\norm{\mu}{\ast}(h'-h)/h.$
    We will now show that
    \begin{equation}
      \label{eq:DifferenceParallelepipedDilated}
      \left| \frac{\mu(l_+)}{h^{d-1}h'} - \frac{\mu(L_+)}{2^dh^{d-1}h'} \right| \leq C_d \norm{\mu}{\ast} \frac{h'-h}{h} \left( 1 + \log\left(\frac{h}{h'-h}+1\right) \right)
    \end{equation}
    The last term is estimated in a similar way.
    First, we use the decomposition of $L_+$ to split the difference as follows
    \begin{equation*}
      \left|\frac{\mu(l_+)}{h^{d-1}h'} - \frac{\mu(L_+)}{2^dh^{d-1}h'}\right| \leq \frac{1}{2^dh^{d-1}h'} \sum_{i=1}^{2^d} |\mu(l_+)-\mu(L_+^i)|.
    \end{equation*}
    For each term in this sum, we can use the estimate in \eqref{eq:DifferenceParallelepiped} for parallelepipeds, just taking into account that now the role of $h$ is taken by $Ch$ and $h'-h$ plays the role of $|x-t|.$
    This gives \eqref{eq:DifferenceParallelepipedDilated}, which yields \eqref{eq:SecondDividedDifferenceParallelepiped} and finishes the proof.
  \end{proof}

  We also need a dyadic version of Theorem \ref{thm:MeasureIBMODistance}.
  We say that $Q$ is a \emph{dyadic cube} in $\reals^d$ if it is of the form $[k_12^{-n},(k_1+1)2^{-n}) \times \ldots \times [k_d2^{-n},(k_d+1)2^{-n})$ where $k_1,\ldots,k_d \in \integers$ and $n \geq 0,$ or if it is of the form $[k_12^{-n}-t_n,(k_1+1)2^{-n}-t_n) \times \ldots \times [k_d2^{-n}-t_n,(k_d+1)2^{-n}-t_n)$ where $k_1,\ldots,k_d \in \integers,$ $n < 0$ and where $t_n$ is the quantity defined in Section \ref{sec:Dyadic}.
  We denote here the set of dyadic cubes in $\reals^d$ by $\mathcal{D}$ and the set of dyadic cubes of side length $2^{-n}$ by $\mathcal{D}_n.$
  As we did before, if $Q_0$ is a given arbitrary cube, we may refer to the set of dyadic cubes contained in $Q_0$ by $\mathcal{D}(Q_0).$
  For future convenience, given a signed Borel measure $\mu$ on $\reals^d,$ we define the \emph{dyadic second divided difference} as
  \begin{equation*}
    \Delta_2^d\mu(Q) = \Delta_1\mu(Q) - \Delta_1\mu(Q^\ast), \quad Q \in \mathcal{D},
  \end{equation*}
  where we used $Q^\ast$ to denote the unique dyadic cube that contains $Q$ and is such that $l(Q^\ast) = 2l(Q).$
  We will also need the \emph{maximal dyadic second divided difference,} defined by
  \begin{equation*}
    \Delta_2^\ast\mu(Q) = \max_{Q'} |\Delta_1\mu(Q') - \Delta_1\mu(Q)|, \quad Q \in \mathcal{D},
  \end{equation*}
  where $Q'$ ranges over all dyadic cubes contained in $Q$ such that $l(Q') = l(Q)/2.$
  A signed Borel measure $\mu$ on $\reals^d$ is called a \emph{dyadic Zygmund measure,} $\mu \in \zygdyc,$ if
  \begin{equation*}
    \norm{\mu}{\ast d} = \sup_{Q \in \mathcal{D}} \Delta_2^\ast\mu(Q) < \infty.
  \end{equation*}
  A real valued function $f$ on $\reals^d$ is said to have \emph{bounded dyadic mean oscillation} on $\reals^d,$ $f \in \BMO_d(\reals^d),$ if
  \begin{equation*}
    \norm{f}{\BMO d} = \sup_{Q \in \mathcal{D}} \left(\frac{1}{|Q|} \int_Q |f(x) - f_Q|^2 \, dx\right)^{1/2} < \infty.
  \end{equation*}
  We will say that a signed Borel measure $\nu$ on $\reals^d$ is a \emph{dyadic $\IBMO$ measure,} $\nu \in \IBMO_d,$ if it is absolutely continuous and its derivative is
  \begin{equation*}
    d\nu = b(x)\, dx,
  \end{equation*}
  where $b \in \BMO_d(\reals^d).$
  It can be checked that $\nu$ is such a measure if and only if it satisfies
  \begin{equation*}
    \sup_{Q \in \mathcal{D}} \left(\frac{1}{|Q|} \sum_{R \in \mathcal{D}(Q)} |\Delta_2^d\nu(R)|^2|R|\right)^{1/2} < \infty.
  \end{equation*}
  The analog of Theorem \ref{thm:MeasureIBMODistance} for these dyadic spaces is the following.
  
  \begin{thm}
    \label{thm:MeasureIBMODyadic}
    Let $\mu$ be a compactly supported measure in $\zygdyc.$
    For each $\varepsilon > 0$ consider
    \begin{equation*}
      D(\mu,\varepsilon) = \sup_{Q \in \mathcal{D}} \frac{1}{|Q|} \sum_{\substack{R \in \mathcal{D}(Q) \\ \Delta_2^\ast\mu(R^\ast) > \varepsilon}} |R|.
    \end{equation*}
    Then,
    \begin{equation}
      \label{eq:MeasureDistanceIBMODyadic}
      \dist(\mu,\IBMO_d) = \inf \{\varepsilon > 0 \colon D(\mu,\varepsilon) < \infty\}.
    \end{equation}
  \end{thm}

  Note that, as we did for functions on $\reals,$ we can rewrite this result in terms of dyadic martingales on $\reals^d.$
  We define a \emph{dyadic martingale on $\reals^d$} as a sequence of functions $S = \{S_n\}_{n=0}^\infty$ such that $S_n$ is constant on any cube $Q \in \mathcal{D}_n$ and such that
  \begin{equation*}
    S_n | Q = \frac{1}{2^d} \sum_{\substack{Q' \in \mathcal{D}_{n+1}\\ Q' \subset Q}} S_{n+1} | Q',
  \end{equation*}
  for all $Q \in \mathcal{D}_n,$ $n \geq 0.$
  Given a measure $\mu \in \zyg,$ we can define a dyadic martingale by taking
  \begin{equation}
    \label{eq:DensitiesMartingale}
    S_n(Q) = \Delta_1\mu(Q), \quad Q \in \mathcal{D}_n, n \geq 0,
  \end{equation}
  and then $\Delta S(Q) = S_n(Q) - S_{n-1}(Q^\ast) = \Delta_2^d\mu(Q),$ for $Q \in \mathcal{D}_n,$ and we can rewrite Theorem \ref{thm:MeasureIBMODyadic} in terms of martingales.
  Following this relation between dyadic second divided differences for measures and martingale jumps, we will denote $\Delta^\ast S(Q) = \Delta_2^\ast\mu(Q).$

  \begin{proof}[Proof of Theorem \ref{thm:MeasureIBMODyadic}]
    Assume that $\mu$ is supported on the unit cube $Q_0 = [0,1]^d.$
    We need to prove that, for a given $\varepsilon > 0,$ there is a measure $\nu \in \IBMO_d$ satisfying $\norm{\mu-\nu}{\ast d} \leq \varepsilon$ if and only if $D(\mu,\varepsilon) < \infty.$
    Denote by $\varepsilon_0$ the infimum in the left-hand side of \eqref{eq:MeasureDistanceIBMODyadic}.

    Given $\varepsilon > \varepsilon_0,$ consider the martingale $S$ defined by \eqref{eq:DensitiesMartingale}.
    Approximate the martingale $S$ by another dyadic martingale $B$ in the following way.
    Start taking $B(Q_0) = S(Q_0).$
    Then, for $Q \in \mathcal{D}(Q_0),$ set $\Delta B(Q) = \Delta S(Q)$ whenever $\Delta^\ast S(Q^\ast) > \varepsilon,$ and set $\Delta B(Q) = 0$ otherwise.
    By construction, it is clear that $|\Delta S(Q) - \Delta B(Q)| \leq \varepsilon$ for any dyadic cube $Q.$
    Moreover, for any such cube $Q,$ we have that
    \begin{equation}
      \label{eq:IBMOMeasureCondition}
      \frac{1}{|Q|} \sum_{R \in \mathcal{D}(Q)} |\Delta B(R)|^2|R| = \frac{1}{|Q|} \sum_{\substack{R \in \mathcal{D}(Q) \\ \Delta^\ast S(R^\ast) > \varepsilon}} |\Delta S(R)|^2 |R| \lesssim \norm{\mu}{\ast} D(\mu,\varepsilon).
    \end{equation}
    Define now $b(x) = \lim_n B_n(x) = \sum_{n=1}^\infty \Delta B_n(x).$
    Using that, for any dyadic martingale, the increments $\Delta B_j$ are $\Lp{2}$ orthogonal, we get that
    \begin{equation*}
      \int_{Q_0} b(x)^2\, dx = \int_{Q_0} \sum_{n=1}^\infty |\Delta B_n(x)|^2\, dx = \sum_{R \in \mathcal{D}(Q_0)} |\Delta B(R)|^2 |R| < \infty,
    \end{equation*}
    so that $b \in \Lp{2}$ and it is finite almost everywhere.
    Hence, the measure $\nu$ defined by
    \begin{equation*}
      d\nu = b(x)\, dx,
    \end{equation*}
    is an absolutely continuous measure that, by \eqref{eq:IBMOMeasureCondition}, is an $\IBMO_d$ measure such that $\norm{\mu-\nu}{\ast d} \leq \varepsilon.$

    On the other hand, if $\varepsilon < \varepsilon_0,$ there exists no measure $\nu \in \IBMO_d$ satisfying $\norm{\mu-\nu}{\ast d} \leq \varepsilon.$
    Indeed, take $\varepsilon < \varepsilon_1 < \varepsilon_0$ and assume that there is $\nu \in \IBMO_d$ such that $\norm{\mu-\nu}{\ast d} \leq \varepsilon.$
    Then, for any $Q \in \mathcal{D}$ such that $\Delta_2^\ast\mu(Q^\ast) > \varepsilon_1,$ we have that $\Delta_2^\ast\nu (Q^\ast) > \varepsilon_1 - \varepsilon = \delta > 0.$
    Thus
    \begin{equation*}
      \frac{1}{|Q|} \sum_{R \in \mathcal{D}(Q)} |\Delta_2^d\nu(R)|^2 |R| \geq \frac{\delta^2}{|Q|} \sum_{\substack{R \in \mathcal{D}(Q) \\ \Delta_2^\ast\mu(R^\ast)>\varepsilon_1}} |R|,
    \end{equation*}
    but the supremum over $Q \in \mathcal{D}$ of this last quantity is $\delta^2 D(\mu,\varepsilon_1),$ which is infinite since $\varepsilon_1 < \varepsilon.$
    This contradicts that $\nu \in \IBMO_d.$
  \end{proof}

  The proof of Theorem \ref{thm:MeasureIBMODistance} follows the same lines than the proof of Theorem \ref{thm:IBMODistance}.
  We just mention that the construction used to prove Theorem \ref{thm:zygAverages} is easily adapted to the setting of $\reals^d,$ except for the following detail.
  Let $Q$ be a cube in $\reals^d$ and consider the covering $\mathcal{F}(Q) = \{R_j\}$ of $Q$ by maximal dyadic cubes, in the same sense as we did in $\reals.$
  In the case $d=1$ we could have at most two elements of the same size in $\mathcal{F}(Q),$ but this does not hold for $d \geq 2.$
  For $d \geq 2,$ the amount of cubes $R_j$ in $\mathcal{F}(Q)$ of size $|R_j| = 2^{-kd}|Q|,$ for some $k \geq 1,$ is of the order of $2^{k(d-1)}.$
  Using this bound, one sees that the sums appearing in the estimates in the proof of Theorem \ref{thm:zygAverages} are convergent and bounded by a universal constant.
  
  \section{An Application to Sobolev Spaces}
  \label{sec:Sobolev}
  Fix $1 < p < \infty.$
  Consider the Sobolev space $\sobolev{1}{p}$ of functions $f \in \Lp{p}$ whose derivative $f'$ in the sense of distributions is also in $\Lp{p}.$
  Consider as well, in the Zygmund class, the subspace $\zyg^p = \sobolev{1}{p}\cap\zyg.$
  For $x \in \reals,$ consider the truncated cone $\Gamma(x) = \{(t,h) \in \reals^2_{+} \colon |x-t| < h < 1\}.$
  In \cite{ref:Nicolau} it is shown that a function $f \in \Lp{p}$ is in the Sobolev space $\sobolev{1}{p}$ if and only if $C(f) \in \Lp{p},$ where
  \begin{equation*}
    C(f)(x) = \left( \int_{\Gamma(x)} |\Delta_2f(s,t)|^2 \frac{ds\, dt}{t^2} \right)^{1/2},\quad x \in \reals.
  \end{equation*}
  
  The purpose of this section is to prove Theorem \ref{thm:sobolev}.
  Following the same scheme as before, we first need a dyadic version of the previous theorem.
  Let us first recall some more concepts and standard results of Martingale Theory that will be useful later.
  The \emph{quadratic characteristic} of a dyadic martingale $S$ is the function
  \begin{equation*}
    \langle S \rangle (x) = \left( \sum_{n=1}^{\infty} |\Delta S_n(x)|^2 \right)^{1/2},\quad x \in \reals,
  \end{equation*}
  and its \emph{maximal function} is
  \begin{equation*}
    S^\ast (x) = \sup_n |S_n(x)-S_0(x)|, \quad x \in \reals.
  \end{equation*}
  Given $0 < p < \infty$ and a dyadic martingale $S,$ the Burkholder-Davis-Gundy Inequality (see \cite{ref:BanuelosMoore}) states that there exists a constant $C = C(p) > 0$ such that
  \begin{equation}
    \label{eq:BDGInequality}
    C^{-1} \norm{\langle S \rangle}{\Lp{p}} \leq \norm{S^\ast}{\Lp{p}} \leq C \norm{\langle S \rangle}{\Lp{p}}.
  \end{equation}
  Remember as well that the \emph{Fatou set} of a dyadic martingale $S,$ denoted by $F(S),$ is defined as
  \begin{equation*}
    F(S) = \{x \in \reals \colon \lim_n S_n(x) \text{ exists and is finite}\}.
  \end{equation*}
  It is a standard result of Martingale Theory that, for a dyadic martingale $S$ such that $\norm{S}{\ast} < \infty,$ its Fatou set is $F(S) = \{x \in \reals \colon \langle S \rangle (x) < \infty\},$ where the equality must be understood up to sets of zero measure (see \cite{ref:Llorente}).
  
  Using the characterisation for the Sobolev space $\sobolev{1}{p}$ previously stated, we say that a function $b$ is in the dyadic space $\zygdyc^p$ if its average growth martingale $B,$ as defined in \eqref{eq:AGMartingale}, has quadratic characteristic $\langle B \rangle \in \Lp{p}$ and
  \begin{equation*}
    \norm{B}{\ast} = \sup_{I \in \mathcal{D}} |\Delta B(I)| < \infty.
  \end{equation*}
  Note that, in fact, $\zygdyc^p = \sobolev{1}{p}\cap\zygdyc.$
  Indeed, if $b \in \zygdyc^p,$ by definition $b \in \zygdyc.$
  Moreover, since its average growth martingale $B$ has quadratic characteristic $\langle B \rangle \in \Lp{p},$ $\langle B \rangle (x) < \infty$ for almost every $x \in \reals.$
  Thus, $B(x) = \lim_n B_n(x)$ exists almost everywhere and will satisfy $b'(x) = B(x)$ in the sense of distributions.
  Using \eqref{eq:BDGInequality}, $B^\ast \in \Lp{p}$ and, thus, $B \in \Lp{p}$ as well, which is the same to say that $b' \in \Lp{p}.$
  We now state the analogous of Theorem \ref{thm:sobolev} in this context.
  
  \begin{thm}
    \label{thm:sobolevDyadic}
    Let $f$ be a compactly supported function in $\zygdyc$ and fix $1 < p < \infty.$
    Let $S$ be the average growth martingale of $f.$
    For every $\varepsilon > 0,$ define the truncated quadratic characteristic
    \begin{equation*}
      D(f,\varepsilon)(x) = \left(\# \{n\colon |\Delta S_n(x)| > \varepsilon\}\right)^{1/2}.
    \end{equation*}
    Then,
    \begin{equation}
      \label{eq:SobolevDyadicDistance}
      \dist(f,\zygdyc^p) = \inf \{\varepsilon > 0 \colon D(f,\varepsilon) \in \Lp{p}\}.
    \end{equation}
  \end{thm}
  
  \begin{proof}
    Let $\varepsilon_0$ be the infimum on \eqref{eq:SobolevDyadicDistance}.
    Assume $0 < \varepsilon < \varepsilon_1 < \varepsilon_0$ and that there is $b \in \zygdyc^p$ such that $\norm{f-b}{\ast d} \leq \varepsilon.$
    Let $B$ be the average growth martingale of function $b.$
    Whenever $|\Delta S_n(x)| > \varepsilon_1,$ we have that $|\Delta B_n(x)| > \varepsilon_1 - \varepsilon = \delta > 0.$
    Thus,
    \begin{align*}
      \langle B \rangle^2 (x) = \sum_{n=1}^{\infty} |\Delta B_n(x)|^2 &\geq
      \sum_{|\Delta B_n(x)|>\delta} |\Delta B_n(x)|^2 \\
      &\geq \frac{\delta^2}{\norm{f}{\ast d}^2} D^2(f,\varepsilon_1)
    \end{align*}
    for all $x \in \reals.$
    But, since $\varepsilon_1 < \varepsilon_0,$ $D(f,\varepsilon_1) \not\in \Lp{p}$ and so $\langle B \rangle \not\in \Lp{p},$ getting in this way a contradiction.
    Hence, we see that $\dist(f,\zygdyc^p) \geq \varepsilon_0.$
    
    Assume that $f$ is supported on $I_0.$
    Consider now $\varepsilon > \varepsilon_0.$
    Construct a dyadic martingale $B$ with $B(I_0) = S(I_0)$ and such that $\Delta B(I) = \Delta S(I)$ for all $I \in \mathcal{D}(I_0)$ whenever $|\Delta S(I)| > \varepsilon,$ but take $\Delta B(I) = 0$ when $|\Delta S(I)| \leq \varepsilon.$
    Note that $\langle B \rangle \in \Lp{p}.$
    Therefore, using \eqref{eq:BDGInequality}, we see that we can define $b'(x) = \lim_n B_n(x)$ almost everywhere with $b' \in \Lp{p}.$
    Taking now $b(x) = \int_0^x b'(s)\, ds,$ we get $b \in \zygdyc^p$ such that $\norm{f-b}{\ast d} \leq \varepsilon.$
    This shows that $\dist(f,\zygdyc^p) \leq \varepsilon_0,$ completing the proof.
  \end{proof}
  
  \begin{proof}[Proof of Theorem \ref{thm:sobolev}]
    Let $\varepsilon_0$ be the infimum in \eqref{eq:SobolevDistance}.
    Assume $0 < \varepsilon < \varepsilon_1 < \varepsilon_0,$ take $\delta = \varepsilon_1 - \varepsilon,$ and assume that there is $b \in \zyg^p$ such that $\norm{f-b}{\ast} \leq \varepsilon.$
    The same argument used in the first part of the proof of Theorem \ref{thm:sobolevDyadic} allows us to see that
    \begin{equation*}
      C(b)(x) \geq \delta C(b,\delta)(x) \geq \delta C(f,\varepsilon_1)(x)
    \end{equation*}
    for $x \in \reals.$
    Since $\varepsilon_1 < \varepsilon_0,$ we have that $C(f,\varepsilon_1) \not\in \Lp{p}$ and, thus, $C(b) \not\in \Lp{p},$ contradicting that $b \in \zyg^p.$
    Hence, $\dist(f,\zyg^p) \geq \varepsilon_0.$
    
    Fix $\varepsilon > \varepsilon_0$ so that $C(f,\varepsilon) \in \Lp{p}.$
    For $\alpha \in [-1,1],$ consider $f^{(\alpha)} = f(x+\alpha).$
    Note that $C(f^{(\alpha)},\varepsilon) \in \Lp{p}$ as well.
    Using the same argument as in the proof of Theorem \ref{thm:IBMODistance}, one can see that this fact implies that $D(f^{(\alpha)},\varepsilon) \in \Lp{p}.$
    Thus, for each $\alpha \in [-1,1],$ the function $f^{(\alpha)}$ satisfies the hypothesis of Theorem \ref{thm:sobolevDyadic} and may be approximated as $f^{(\alpha)} = b^{(\alpha)} + t^{(\alpha)},$ where $b^{(\alpha)} \in \zygdyc^p$ with $\norm{b^{(\alpha)}}{\ast d} \leq \norm{f}{\ast}$ and $\norm{t}{\ast d} \leq \varepsilon.$
    Apply now Theorem \ref{thm:zygAverages}, with $R = 1,$ both with the mapping $\alpha \mapsto b^{(\alpha)}$ and $\alpha \mapsto t^{(\alpha)}$ to obtain respectively functions $b$ and $t$ such that $f = b + t$ and such that $b \in \zyg^p$ and $\norm{t}{\ast} \lesssim \varepsilon.$
    This completes the proof.
  \end{proof}
  
  \section{Open Problems}
  \label{sec:OpenProblems}
  This last section is devoted to state three open problems closely related to our results.
  
  \textbf{1.} Observe that Theorem \ref{thm:MeasureIBMODistance} is a generalisation of Theorem \ref{thm:IBMODistance} for measures on $\reals^d$ that works for any $d \geq 1.$
  Nonetheless, we have not been able to generalise Theorem \ref{thm:IBMODistance} for functions on the Zygmund class on $\reals^d$ for $d > 1.$
  We say that a continuous function $f\colon \reals^d \rightarrow \reals$ is in the Zygmund class $\zyg(\reals^d)$ if
  \begin{equation*}
    \norm{f}{\ast} = \sup_{\substack{x,h \in \reals^d\\ h \neq 0}} \frac{|f(x+h)-2f(x)+f(x-h)|}{\norm{h}{}} < +\infty.
  \end{equation*}
  A continuous function on $\reals^d$ is in $\IBMO(\reals^d)$ if its partial derivatives in the sense of distributions are functions in $\BMO.$
  It is easy to see that $\IBMO(\reals^d) \subset \zyg(\reals^d).$
  We do not know an analog of Theorem \ref{thm:IBMODistance} when $d > 1.$
  Roughly speaking, the method we develop in this paper works discretising a function in the Zygmund class and modifying its average growth on certain intervals, so that we end up constructing the derivative of a function in $\IBMO.$
  Nonetheless, when applying this method in $d > 1$ variables, one ends up constructing $d$ functions that approximate the divided differences of a given function in the coordinate directions.
  However, this system of $d$ functions is not, in general, the gradient of an $\IBMO$ function.

  \medskip

  \textbf{2.} Another related open problem is to find the closure of the space of Lipschitz functions, $\Lip,$ in the Zygmund class.
  It is a well known fact that singular integral operators such as the Hilbert transform are bounded in $\Lp{p}$ for $1 < p < \infty,$ but they are not in $\Lp{\infty}.$
  Nonetheless, they are bounded from $\Lp{\infty}$ to $\BMO$ and from $\BMO$ to itself.
  In a similar fashion, these operators are bounded on the Hölder classes $\Lip_\alpha,$ for $0 < \alpha < 1,$ but they are not in the Lipschitz class.
  However, they are actually bounded on the Zygmund class, which plays a similar role to that of $\BMO$ in the previous setting.
  In this sense, this problem would be related to the one solved by Garnett and Jones in \cite{ref:GarnettJonesDistanceInBMO}, where they find a characterisation for the closure of $\Lp{\infty}$ in the space $\BMO.$

  \medskip
  
  \textbf{3.} The open problem mentioned in the previous paragraph is analogous to the well known open problem of describing the closure of the space $\mathbb{H}^\infty$ of bounded analytic functions in the unit disk into the Bloch space $\mathcal{B},$ consisting of analytic functions $f$ in the unit disk such that $\sup_{z \in \disk} (1-|z|^2) |f'(z)| < \infty$ (see \cite{ref:AndersonCluniePommerenke}).
  In \cite{ref:GhatageZheng} the closure of the space $\operatorname{BMOA}$ in $\mathcal{B}$ is described, while in \cite{ref:MonrealNicolau} and \cite{ref:GalanopoulosMonrealPau} the closure of the class $\mathbb{H}^p\cap\mathcal{B}$ is studied.
  Here $\mathbb{H}^p$ denotes the classical Hardy spaces of analytic functions in the unit disk.
  It is worth mentioning that in this setting, the proofs rely on reproducing formulae that analytic functions fulfill, which are not available in our real variable situation.
  
  \printbibliography
  
  \Addresses

\end{document}